\documentclass[leqno,CJK]{siamltex704}
\usepackage{amssymb,amsmath,graphicx,amscd,mathrsfs}
\usepackage{color,xcolor,amsmath}
\usepackage{amsmath}
\usepackage{graphicx}
\usepackage{mathrsfs}
\usepackage{float}
\usepackage{amsfonts,amssymb}
\usepackage{dsfont}
\usepackage{pifont}
\usepackage{hyperref}
\usepackage{multirow}
\usepackage{placeins}
\numberwithin{equation}{section}
\def\3bar{{|\hspace{-.02in}|\hspace{-.02in}|}}
\def\E{{\mathcal{E}}}
\def\T{{\mathcal{T}}}

\def\dQ{{\mathbb{Q}}}

\def\b0{\boldsymbol{0}}
\def\sumT{\sum_{T\in\mathcal{T}_h}}     %new
\def\bn{{\mathbf{n}}}

\def\bf{{\mathbf{f}}}
\def\bq{{\mathbf{q}}}

\newtheorem{algorithm1}{Weak Galerkin Algorithm}

\def\tQ{\tilde Q}
 \newcommand{\Real}{\mathbb{R}}

 \newcommand{\trb}[1]{|\!|\!|#1|\!|\!|}
 \newcommand{\ba}{\bar a}

\allowdisplaybreaks % For long formula

\title{A weak Galerkin finite element scheme with boundary
continuity for second-order elliptic problems} %
\author{Qilong Zhai\thanks{Department of Mathematics,
Jilin University, Changchun China} \and Xiu Ye\thanks{Department of
Mathematics, University of Arkansas at Little Rock, Little Rock, AR
72204 (xxye@ualr.edu). The research of Ye was supported in part by
National Science Foundation Grant DMS-1115097.} \and Ruishu
Wang\thanks{Department of Mathematics, Jilin University, Changchun
China} \and Ran Zhang\thanks{Department of Mathematics, Jilin
University, Changchun, China (zhangran@mail.jlu.edu.cn). The
research of Zhang was supported in part by China Natural National
Science Foundation(11271157, 11371171, 11471141), and by the Program
for New Century Excellent Talents in University of Ministry of
Education of China.} }

\begin{document}

\maketitle

\begin{abstract}
A new weak Galerkin (WG) finite element method for solving the
second-order elliptic problems on polygonal meshes by using
polynomials of boundary continuity is introduced and analyzed. The
WG method is utilizing weak functions and their weak derivatives
which can be approximated by polynomials in different combination of
polynomial spaces. Different combination gives rise to different
weak Galerkin finite element methods, which makes WG methods highly
flexible and efficient in practical computation. This paper explores
the possibility of certain combination of polynomial spaces that
minimize the degree of freedom in the numerical scheme, yet without
losing the accuracy of the numerical approximation. Error estimates
of optimal order are established for the corresponding WG
approximations in both a discrete $H^1$ norm and the standard $L^2$
norm. In addition, the paper also presents some numerical
experiments to demonstrate the power of the WG method. The numerical
results show a great promise of the robustness, reliability,
flexibility and accuracy of the WG method.
\end{abstract}



\section{Introduction}

Nowadays, finite element methods (FEMs) are widely used in almost
every field of engineering and industrial analysis. Since R. Courant
\cite{Courant43} formulated the essence of what is now called a
finite element in 1943, this method has been getting more and more
attractive with the development of computers and is now recognized
as one of the most versatile and powerful methods for approximating
the solutions of boundary-value problems, especially for problems
over complicated domains. Among the different finite elements
methods, the conforming finite element method \cite{SF73} with
continuous, piecewise polynomial approximating spaces, has long been
employed to approximate solutions for partial differential
equations. Within the past few decades, however, a number of
researchers have investigated Galerkin methods based on fully
discontinuous approximating spaces, such as the discontinuous
Galerkin (DG) methods \cite{CKS00, ABP01, ABCM00, BH00}, the
Hybridized discontinuous Galerkin (HDG) methods \cite{CGL09, KSC12},
the weak Galerkin (WG) methods \cite{WY13, WY14}, etc.

The WG method was first introduced in 2012 \cite{WY13} for the
second order elliptic problem and further developed with other
applications, such as the Stokes, Helmholtz, Maxwell, biharmonic
\cite{mwy, mwy0927, mwwy0618, mwy3818, mwyz-maxwell,
mwyz-biharmonic, WY13, WY14, ZhangBH}, etc. Its central idea is that
the shape function in the interior of each element is simply
polynomial. The weak finite element function could be totally
discontinuous across elements. The continuity is compensated by the
stabilizer through a suitable boundary integral defined on the
boundary of elements. That is, we know more information on the shape
function by sacrificing the continuity.

Comparing with the conforming FEMs, there are a lot of advantages
for the discontinuous methods, for instance, high-order accuracy,
multiphysics capability, the finite element partition can be of
polygon or polyhedral type, the weak finite element space is easy to
construct with any approximation requirement and suit for any given
stability requirement. All of this, however, comes at a price: most
notably through an increase in the total degrees of freedom (d.o.f.)
as a direct result of the decoupling of the elements. For linear
elements, this yields a doubling in the total number of degrees of
freedom compared to the continuous FEM. For WG finite element
methods, in order to reduce the d.o.f., people have made some
efforts: in \cite{MWY15, ZhangZhai15}, the possibility of optimal
combination of polynomial spaces was explored to minimize the number
of unknowns in the numerical scheme, yet without compromising the
accuracy of the numerical approximation; in \cite{ZhangZhai15}, some
hybridization of finite element methods has been introduced by
utilizing the Lagrange multiplier. The distinctive feature of the
methods in this framework is that the only globally coupled degrees
of freedom are those of an approximation of the solution defined
only on the boundaries of the elements, then the global unknowns are
the numerical traces of the field variables. Thus, one can reduce
the number of the globally coupled degrees of freedom of WG methods.

The goal of this paper is to explore the possibility of certain
combination of the polynomial spaces to reduce the number of
unknowns without compromising the rate of convergence, which
utilizes the boundary continuity for the WG finite element spaces.
This idea is motivated by Chen \cite{Chen2015} that was presented in
ICIAM 2015. Combining with the Schur complement technique, we
further eliminate the interior unknowns and produce a much reduced
system of linear equations involving only the unknowns representing
the interface variables. In fact, if we take the triangle mesh, the
d.o.f. of the new scheme is the same with that of conforming FEMs.

Next, we introduce this WG methods. For the sake of simplicity and
easy presentation of the main ideas, we restrict ourselves to the
following model problem
\begin{eqnarray}\label{problem-eq1}
-\nabla\cdot(a\nabla u)&=& f \quad\text{ in }\Omega,
\\ \label{problem-eq2}
u&=&0 \quad\text{ on }\partial\Omega,
\end{eqnarray}
where $\Omega$ is an open bounded polygonal domain in
$\mathbb{R}^2$. We assume that $f$ are given, sufficiently smooth
functions, and $a$ is a symmetric $2\times 2$ matrix-valued
function. Suppose there exists a positive number $\lambda$ such that
\begin{eqnarray*}
\xi^t a \,\xi \ge \lambda \xi^t\xi,\quad\forall \xi\in\Real^2,
\end{eqnarray*}
where $\xi$ is a column vector and $\xi^t$ is the transpose of
$\xi$.

The paper is organized as follows. In Section 2, we present some
standard notations in Sobolev spaces and preliminaries. A
weakly-defined differential operator is also introduced. The weak
Galerkin finite element scheme is developed and some properties for
the error analysis are discussed in Section 3. In Section 4, we
shall derive an error equation for the WG approximations.
Optimal-order error estimates of $H^1$ and $L^2$ for the WG finite
element approximations are also derived in this Section. The
equivalence of WG formulation and its Schur complement formulation
is proved in Section 5. In Section 6, numerical experiments are
conducted. Finally, we present some technical estimates for
quantities related to the local $L^2$ projections into various
finite element spaces and some approximation properties which are
useful in the convergence analysis in ¡°Appendix¡±.

\section{Notations and preliminaries}

In this section, we shall introduce some notations used in this paper.

We use the standard Soblev space notations. For an open set $D\in
\Real^d$, $\|\cdot\|_{s,D}$ and $(\cdot,\cdot)_{s,D}$ stand for the
$H^s(D)$ norm and inner-product, namely. We shall drop the
subscripts when $s=0$ and $D=\Omega$.
%When the region $D$ is an
%interface in $\Real^{d-1}$, we also use $\langle \cdot,\cdot\rangle$
%to represent the inner-product of the boundary.

Let $\T_h$ be a partition of the domain consisting of polygons in
two dimension or polyhedra in three dimension satisfying a set of
conditions specified in \cite{WY14}. Denote by $\E_h$ the set of all
edges or at faces in $\T_h$. For every element $T \in \T_h$, we
denote by $h_T$ its diameter and mesh size $h = \max_{T\in\T_h} h_T$
for $\T_h$.

For a given partition $\T_h$, $P_k(T)$ denotes a piecewise polynomial on $\T_h$
whose degree is no more than $k$ on each $T\in\T_h$. Similarly, $P_k(e)$ denotes a piecewise polynomial on $\E_h$
whose degree is no more than $k$ on each $e\in\E_h$.

Now, we define the weak finite element space as follows:
\begin{eqnarray*}
V_h=\{(v_0,v_b): v_0|_T\in P_k(T), v_b|_e\in P_k(e), v_b \text{ is
continuous on }\E_h, v_b=0 \text{ on }\partial\Omega \}.
\end{eqnarray*}
It should be noticed that $v_b$ is single-valued on each edge, and
$v_b$ is continuous on $\E_h$, which means that $v_b$ share the same
value on each node of $\T_h$.

Similar to the definition in $\cite{WY13}$, we can define the
following weak gradient operator on $V_h$.
\begin{definition}For any $v_h\in V_h$, define the discrete weak gradient
$\nabla_w v_h|_T \in [P_{k-1}(T)]^2$ satisfying
\begin{eqnarray}\label{wgradient}
(\nabla_w v_h,\bq)_T=-(v_0,\nabla\cdot\bq)_T+\langle
v_b,\bq\cdot\bn\rangle_{\partial T}, \quad\forall \bq\in
[P_{k-1}(T)]^2,
\end{eqnarray}
where $\textbf{n}$ is the outward unit normal vector along $\partial
T$, $(\cdot, \cdot)_T$ stands for the $L^2$-inner product in
$L^2(T)$, and $\langle\cdot, \cdot\rangle_{\partial T}$ is the inner
product in $L^2(\partial T)$
\end{definition}

\section{A weak Galerkin finite element scheme}
In this section, we shall propose a WG scheme for the second-order
elliptic problem, and verify the wellposedness of the scheme.

For any $v_h, w_h\in V_h$, define the following bilinear forms
\begin{eqnarray*}
s(v_h,w_h)&=&\sumT h_T^{-1}\langle v_0-v_b,w_0-w_b\rangle_{\partial T},
\\
a_w(v_h,w_h)&=&(a\nabla_w v_h,\nabla_w w_h)+s(v_h,w_h).
\end{eqnarray*}

\begin{algorithm1}
The weak Galerkin numerical solution of problem (\ref{problem-eq1})
-(\ref{problem-eq2}) can be obtained by seeking $u_h\in V_h$ such that
\begin{eqnarray}\label{wg-scheme}
a_w(u_h,v_h)=(f,v_0)\quad\forall v_h\in V_h.
\end{eqnarray}
\end{algorithm1}

The following semi-norm can be inducted from $a_w(\cdot,\cdot)$
directly
\begin{eqnarray*}
\trb{v_h}^2=a_w(v_h,v_h),\quad\forall v_h\in V_h.
\end{eqnarray*}
We claim that $\trb{\cdot}$ defines a norm on $V_h$ indeed.

Notice that when $\trb{v_h}=0$, we have $\nabla_w v_h=0$ on each
$T\in\T_h$ and
$v_0=v_b$ on each $e\in\E_h$, and it follows that $\forall T\in\T_h$,
\begin{eqnarray*}
0&=&(\nabla_w v_h,\nabla v_0)_T
\\
&=&-(v_0,\nabla\cdot\nabla v_0)_T+\langle v_b,\nabla v_0\cdot\bn
\rangle_{\partial T}
\\
&=&(\nabla v_0,\nabla v_0)_T-\langle v_0-v_b,\nabla v_0\cdot\bn
\rangle_{\partial T}
\\
&=&\|\nabla v_0\|_T^2,
\end{eqnarray*}
which implies that $v_0$ is constant on $T$. With the condition that
$v_b=v_0$ on each $e\in\E_h$ and $v_b=0$ on $\partial\Omega$, we can
conclude that $v_h=0$, and thus $\trb{\cdot}$ defines a norm
on $V_h$. From this property we can arrive the wellposedness of WG
scheme (\ref{wg-scheme}) directly.

\begin{theorem}
The WG scheme (\ref{wg-scheme}) has a unique solution.
\end{theorem}

Next, we shall discuss some properties for the error analysis.

Define the following $L^2$ projections operators:
\begin{eqnarray*}
&&Q_0: L^2(T)\rightarrow P_k(T),\quad \forall T\in\T_h,
\\
&&Q_b: L^2(e)\rightarrow P_k(e),\quad \forall e\in\E_h,
\\
&&\dQ_h: [L^2(T)]^2\rightarrow [P_{k-1}(T)]^2,\quad \forall T\in\T_h.
\end{eqnarray*}

It is known that on a one-dimension edge $e$, for any smooth function
$w$ we can get a polynomial interpolation of degree $k$ with $k+1$
different points. Suppose the $k+1$ interpolation points include the
two endpoints of the edge, and we have the interpolation
operator $I_b$.
%\begin{eqnarray*}
%I_b: C^0(e)\rightarrow P_k(e),\quad\forall e\in\E_h.
%\end{eqnarray*}

Combining $Q_0$ and $I_b$, we can define
\begin{eqnarray*}
\tilde Q_h=\{Q_0,I_b\}:H^1(\Omega)\rightarrow V_h.
\end{eqnarray*}

In the rest of this paper, we denote $\bar a$ the local
$P_0(T)$ projection of $a$. Obviously, $\ba$ is symmetric and
bounded.
With these projections, it arrives at the following community
property.
\begin{lemma}\label{commu-prop}
For any $\bq\in [P_{k-1}(T)]^2$, $w\in H^1(\Omega)$,
\begin{eqnarray*}
(\nabla_w\tQ_h w,\bq)
&=&(\dQ_h \nabla w,\bq)+\sumT\langle I_b w- w,\bq\cdot\bn\rangle
_{\partial T}.
\end{eqnarray*}
\end{lemma}
\begin{proof}
For any $T\in\T_h$, it follows the definition (\ref{wgradient}) and
the integration by parts that\newpage
\begin{eqnarray*}
&&(\nabla_w\tQ_h w,\bq)_T
\\
&=&-(Q_0 w,\nabla\cdot \bq)_T+\langle I_b w,\bq\cdot\bn\rangle
_{\partial T}
\\
&=&-(Q_0 w,\nabla\cdot \bq)_T+\langle w,\bq\cdot\bn\rangle
_{\partial T}+\langle I_b w- w,\bq\cdot\bn\rangle _{\partial T}
\\
&=&(\nabla w,\bq)_T+\langle I_b w- w,\bq\cdot\bn\rangle
_{\partial T}
\\
&=&(\dQ_h \nabla w,\bq)_T+\langle I_b w- w,\bq\cdot\bn\rangle
_{\partial T}.
\end{eqnarray*}
Summing over all $T\in\T_h$ and the proof is completed.
\end{proof}

\section{Error analysis}
In this section, we shall present the $H^1$ and $L^2$ errors of the
weak Galerkin numerical scheme (\ref{wg-scheme}) with optimal
orders.

Denote the error $e_h=\tQ_h u-u_h$, and we shall show the error
equation for $e_h$.
\begin{lemma}\label{error-eqn}
Suppose $u\in H^2(\Omega)$ is the solution of (\ref{problem-eq1})-
(\ref{problem-eq2}), and $u_h$ is the WG numerical solution
of (\ref{wg-scheme}).
Then for any $v_h\in V_h$,
\begin{eqnarray*}
&&a_w(\tQ_h u-u_h,v_h)
=l(u,v_h),
\end{eqnarray*}
where
\begin{eqnarray*}
l(u,v_h)
&=&(a\nabla_w \tQ_h u-\dQ_h (a\nabla u),\nabla_w v_h)
\\
&&-\sumT\langle
(\dQ_h(a\nabla u)-a\nabla u)\cdot\bn,v_0-v_b\rangle_{\partial T}
+s(\tQ_h u,v_h).
\end{eqnarray*}
\end{lemma}
\begin{proof}
For any $v_h\in V_h$, it follows the integration by
parts that
\begin{eqnarray} \nonumber
a_w(u_h,v_h) &=& (f,v_0)
\\ \nonumber
&=&
(a\nabla u,\nabla v_0)-\sumT \langle a\nabla u\cdot\bn,
v_0\rangle_{\partial T}
\\ \label{err-est3}
&=&
(a\nabla u,\nabla v_0)-\sumT \langle a\nabla u\cdot\bn,
v_0-v_b\rangle_{\partial T}.
\end{eqnarray}
From the definition (\ref{wgradient}), we can obtain
\begin{eqnarray}\nonumber
&&(a\nabla u,\nabla v_0)
\\ \nonumber
&=&(\dQ_h (a\nabla u),\nabla v_0)
\\ \label{err-est12}
&=&(\dQ_h (a\nabla u),\nabla_w v_h)+\sumT\langle
\dQ_h(a\nabla u)\cdot\bn,v_0-v_b\rangle_{\partial T}
\end{eqnarray}
Substituting (\ref{err-est12}) into (\ref{err-est3}),
we get
\begin{eqnarray}\label{err-est13}
a_w(u_h,v_h)&=&(\dQ_h (a\nabla u),\nabla_w v_h)+\sumT\langle
(\dQ_h(a\nabla u)-a\nabla u)\cdot\bn,v_0-v_b\rangle_{\partial T}.
\end{eqnarray}
Notice that
\begin{eqnarray}\label{err-est14}
a_w(\tQ_hu,v_h)&=&(a\nabla_w \tQ_h u,\nabla_w v_h)+s(\tQ_h u,v_h).
\end{eqnarray}
Subtracting (\ref{err-est13}) from (\ref{err-est14}), we have
\begin{eqnarray*}
a_w(\tQ_hu-u_h,v_h)&=&(a\nabla_w \tQ_h u-\dQ_h (a\nabla u),\nabla_w v_h)
\\
&&-\sumT\langle
(\dQ_h(a\nabla u)-a\nabla u)\cdot\bn,v_0-v_b\rangle_{\partial T}
+s(\tQ_h u,v_h),
\end{eqnarray*}
which completes the proof.
\end{proof}

With the error equation and the technique tools in Appendix, we can arrive
at the estimate for the $H^1$ error.
\begin{theorem}
Suppose $u\in H_0^1(\Omega)\cap H^{k+1}(\Omega)$ is the solution of (\ref{problem-eq1})-
(\ref{problem-eq2}), and $u_h\in V_h$ is the WG numerical solution
of (\ref{wg-scheme}). The following estimate holds true,
\begin{eqnarray*}
\trb{\tQ_h u-u_h}\le C h^k\|u\|_{k+1}.
\end{eqnarray*}
\end{theorem}
\begin{proof}
Take $v_h=\tQ_h u-u_h$ in Lemma \ref{error-eqn}, then from Lemma
\ref{err-est4} we can obtain
\begin{eqnarray*}
\trb{\tQ_h u-u_h}^2&=&l(u,\tQ_h u-u_h)
\\
&\le& Ch^k\trb{\tQ_hu-u_h},
\end{eqnarray*}
which completes the proof.
\end{proof}

Next, we shall show the optimal $L^2$ error order using the well-known
Nistche's argument. Consider the following dual problem:
find
$\varphi\in H_0^1(\Omega)$ satisfying
\begin{eqnarray}\label{dual-eq}
-\nabla\cdot (a\nabla\varphi)=e_0.
\end{eqnarray}
Assume the dual problem (\ref{dual-eq}) has $H^2$ regularity, i.e.
\begin{eqnarray}\label{regularity-assump}
\|\varphi\|_2\le C\|e_0\|.
\end{eqnarray}
\begin{theorem}
Suppose $u\in H_0^1(\Omega)\cap H^{k+1}(\Omega)$ is the solution of (\ref{problem-eq1})-
(\ref{problem-eq2}), $u_h\in V_h$ is the WG numerical solution
of (\ref{wg-scheme}), and the dual problem (\ref{dual-eq}) satisfies
the regularity assumption (\ref{regularity-assump}). The following estimate holds true,
\begin{eqnarray*}
\|Q_0 u-u_0\|\le C h^{k+1}\|u\|_{k+1}.
\end{eqnarray*}
\end{theorem}
\begin{proof}
Testing equation (\ref{dual-eq}) by $e_0$ and we can obtain
\begin{eqnarray} \nonumber
&&\|e_0\|^2
\\ \nonumber
&=&
(a\nabla \varphi,\nabla e_0)-\sumT \langle \nabla \varphi\cdot\bn,
e_0\rangle_{\partial T}
\\ \nonumber
&=&
(\dQ_h(a\nabla \varphi),\nabla e_0)-\sumT \langle \nabla \varphi\cdot\bn,
e_0-e_b\rangle_{\partial T}
\\ \label{err-est5}
&=&
(\dQ_h(a\nabla \varphi),\nabla_w e_h)+\sumT \langle (\dQ_h(a\nabla \varphi)
-a\nabla\varphi)\cdot\bn,
e_0-e_b\rangle_{\partial T}.
\end{eqnarray}
Similar to the derivative of formula (\ref{err-est14}), we have
\begin{eqnarray}\label{err-est6}
a_w(\tQ_h \varphi,e_h)  &=&(a\nabla_w \tQ_h\varphi,\nabla_w
e_h)+s(\tQ_h\varphi,e_h).
\end{eqnarray}
Subtracting (\ref{err-est6}) from (\ref{err-est5}) yields
\begin{eqnarray*}
\|e_0\|^2=a_w(\tQ_h \varphi,e_h)-l(\varphi,e_h).
\end{eqnarray*}
By substituting $v_h=\tQ_h \varphi$ in Lemma \ref{error-eqn}
we get
\begin{eqnarray*}
a_w(e_h,\tQ_h \varphi)=l(u,\tQ_h \varphi),
\end{eqnarray*}
which implies
\begin{eqnarray}\label{err-est7}
\|e_0\|^2=l(u,\tQ_h \varphi)-l(\varphi,e_h).
\end{eqnarray}

From Lemma \ref{err-est4} and (\ref{regularity-assump}), we derive
that
\begin{eqnarray}\nonumber
l(\varphi,e_h)&\le& Ch\trb{e_h}\|\varphi\|_2
\\ \label{err-est10}
&\le& Ch^{k+1}\|u\|_{k+1}\|e_0\|.
\end{eqnarray}
Now we just need to estimate the following three terms of $l(u,\tQ_h
\varphi)$. As to the first term, we split it into two formulas and
estimate them separately as follows
\begin{eqnarray*}
&&(a\nabla_w \tQ_h u-\dQ_h (a\nabla u),\nabla_w \tQ_h\varphi)
\\
&=&(\nabla_w \tQ_h u-\nabla u,a\nabla_w \tQ_h\varphi)
\\
&=&(\nabla_w \tQ_h u-\nabla u,a\nabla_w \tQ_h\varphi -\dQ_h
(a\nabla_w \tQ_h\varphi))
\\
&&+(\nabla_w \tQ_h u-\nabla u,\dQ_h (a\nabla_w \tQ_h\varphi))
\\
&=&(\nabla_w \tQ_h u-\nabla u,a\nabla_w \tQ_h\varphi -\dQ_h
(a\nabla_w \tQ_h\varphi))
\\
&&+(\nabla_w \tQ_h u-\dQ_h\nabla u,\dQ_h (a\nabla_w \tQ_h\varphi)).
\end{eqnarray*}
For the first formula, from Lemma \ref{Qh-dQ} and Lemma \ref{Qh-H2}
we can obtain
\begin{eqnarray*}
&&|(\nabla_w \tQ_h u-\nabla u,a\nabla_w \tQ_h\varphi -\dQ_h
(a\nabla_w \tQ_h\varphi))|
\\
&\le&\|\nabla_w \tQ_h u-\nabla u\| \|a\nabla_w \tQ_h\varphi -\dQ_h
(a\nabla_w \tQ_h\varphi)\|
\\
&\le&Ch^{k}\|u\|_{k+1}\|\nabla_w \tQ_h\varphi\|_1
\\
&\le&Ch^{k+1}\|u\|_{k+1}\|\varphi\|_2.
\end{eqnarray*}
As to the second formula, it follows Lemmas \ref{commu-prop} and
\ref{Qh-dQ} that
\begin{eqnarray*}
&&(\nabla_w \tQ_h u-\dQ_h\nabla u,\dQ_h (a\nabla_w \tQ_h\varphi))
\\
&=&\sumT\langle I_bu-u,\dQ_h (a\nabla_w \tQ_h\varphi)\cdot\bn
\rangle_{\partial T}
\\
&=&\sumT\langle I_bu-u,\dQ_h (a\nabla_w
\tQ_h\varphi-a\nabla\varphi)\cdot\bn \rangle_{\partial T}
\\
&&+\sumT\langle I_bu-u,(\dQ_h
(a\nabla\varphi)-a\nabla\varphi)\cdot\bn \rangle_{\partial T}
\\
&\le&Ch^{k+1}\|u\|_{k+1}\|\varphi\|_2.
\end{eqnarray*}
For the second term, from Lemma \ref{err-est1} and the trace
inequality we can obtain
\begin{eqnarray*}
&&\sumT\langle
(\dQ_h(a\nabla u)-a\nabla u)\cdot\bn,Q_0\varphi-I_b\varphi\rangle_{\partial T}
\\
&\le&C\left(\sumT h_T\|\dQ_h(a\nabla u)-a\nabla u\|_{\partial T}^2\right)
^\frac12
\left(\sumT h_T^{-1}\|Q_0\varphi-\varphi\|_{\partial T}^2\right)
^\frac12
\\
&&+C\left(\sumT \|\dQ_h(a\nabla u)-a\nabla u\|_{T}^2\right)
^\frac12
\left(\sumT h_T^{-1}\|\varphi-I_b\varphi\|_{\partial T}^2\right)
^\frac12
\\
&\le& Ch^{k+1}\|u\|_{k+1}\|\varphi\|_2.
\end{eqnarray*}
Similarly, for the third term we have
\begin{eqnarray*}
&&s(\tQ_hu,\tQ_h\varphi)
\\
&=&\sumT h_T^{-1}\langle Q_0 u-I_b u,Q_0 \varphi-I_b \varphi\rangle_{\partial T}
\\
&=&\sumT h_T^{-1}\langle Q_0 u- u,Q_0 \varphi- \varphi\rangle_{\partial T}
+\sumT h_T^{-1}\langle u-I_b u,Q_0 \varphi- \varphi\rangle_{\partial T}
\\
&&+\sumT h_T^{-1}\langle Q_0 u- u,\varphi- I_b\varphi\rangle_{\partial T}
+\sumT h_T^{-1}\langle u-I_b u,\varphi- I_b\varphi\rangle_{\partial T}
\\
&\le& Ch^{k+1}\|u\|_{k+1}\|\varphi\|_2,
\end{eqnarray*}
which leads to
\begin{eqnarray}\nonumber
l(u,\tQ_h \varphi)&\le& Ch^{k+1}\|u\|_{k+1}\|\varphi\|_2
\\ \label{err-est11}
&\le& Ch^{k+1}\|u\|_{k+1}\|e_0\|.
\end{eqnarray}
By substituting (\ref{err-est10}) and (\ref{err-est11})
into (\ref{err-est7}), the proof is completed.
\end{proof}

\section{Schur Complement}
In this section, we shall show the technique in practical which can
eliminate the degree of freedoms of $v_0$, and then only the degree of freedoms
of $v_b$ are involved in the total linear system.

Let $u_h=\{u_0,u_b\}$ be the solution of the WG scheme (\ref{wg-scheme}),
then $u_h$ satisfies the following equation
\begin{eqnarray*}
a_w(u_h,v_h)=(f,v_0),\quad\forall v_h\in V_h.
\end{eqnarray*}
This equation can be rewritten as the following form
\begin{eqnarray}\label{schur-1}
a_w(u_h,v_h)&=&(f,v_0), \quad\forall v_h=\{v_0,0\}\in V_h,
\\ \label{schur-2}
a_w(u_h,v_h)&=&0      , \quad\forall v_h=\{0,v_b\}\in V_h.
\end{eqnarray}
For given $u_b$ on $\partial T$, one can solve equation
(\ref{schur-1}) for the $u_0$, and note
\begin{eqnarray}\label{schur-3}
u_0=D(u_b,f).
\end{eqnarray}
It should be noticed that $D(u_b,f)$ can be calculated
elementwisely. We can solve the local problem
\begin{eqnarray}\label{schur-6}
a_w(u_h,v_0)&=&(f,v_0), \quad\forall v_0\in P_k(T),
\end{eqnarray}
and get
\begin{eqnarray}\label{schur-5}
u_0=D_T(u_b,f).
\end{eqnarray}
Then we substitute (\ref{schur-3}) into (\ref{schur-2}) and
get the following equation
\begin{eqnarray}\label{schur-4}
a_w(\{D(u_b,f),u_b\},v_h)=0,\quad\forall v_h=\{0,v_b\}\in V_h.
\end{eqnarray}

Thus we can conclude a complement for the WG scheme as follows.

Step 1. Solve equation (\ref{schur-6}) on each element
$T\in\T_h$ and get $u_0=D(u_b,f)$.

Step 2. Solve the global system (\ref{schur-2}) for $u_b$.

Step 3. Recover $u_0$ by $u_0=D_T(u_b,f)$ on each element.

When we take the uniform triangle mesh on the unit square and set
$k=1$, the degree of freedom of different schemes are listed in
Table \ref{table-3}.
\begin{table}[!http]\label{table-3}
\centering \caption{Degree of freedom of different schemes.}
\begin{tabular}{|c|c|c|c|c|c|c|c|c|c|c|c|c|c|}
\hline
$h $    & $WG $  & $CWG$ & $WG~Schur$ & $CWG~schur$ & $CG$ \\
\hline
   1/8   & 592   & 465   & 208   & 81    & 81 \\
\hline
   1/16  & 2336  & 1825  & 800   & 289   & 289 \\
\hline
   1/32  & 9280  & 7233  & 3136  & 1089  & 1089 \\
\hline
   1/64  & 36992 & 28801 & 12416 & 4225  & 4225 \\
\hline
1/128 & 147712 & 114945 & 49408 & 16641 & 16641 \\
\hline
\end{tabular}
\end{table}

$WG$ stands for the degree of freedom in the scheme in \cite{WY2},
$CWG$ stands for scheme (\ref{wg-scheme}), $WG~Schur$ stands for the
scheme in \cite{WY2} with Schur complement, $CWG~Schur$ stands for
scheme (\ref{wg-scheme}) with Schur complement, and $CG$ stands for
the conforming Galerkin method.

\section{Numerical experiment}
In this section, we shall present some numerical results to show
the efficiency and accuracy of the WG scheme.

\textbf{Example 1} Consider the problem (\ref{problem-eq1}) -
(\ref{problem-eq2}) on the unit square $\Omega=[0,1]\times [0,1]$.
The analytic solution $u$ is set to be $u=\sin (\pi x)\sin(\pi y)$,
the right side and the Dirichlet boundary condition is computed to
match the exact solution. In this numerical experiment, the uniform
triangle mesh is employed. We can conclude from Table \ref{table-1}
that the convergence rate for $H^1$ and $L^2$ errors are of $O(h)$
and $O(h^2)$, respectively. This is coincide with the theoretical
analysis in this paper.

\begin{table}[!http]\label{table-1}
\centering \caption{k=1 uniform convergence rates.}
\begin{tabular}{|c|c|c|c|c|c|c|c|c|c|c|c|c|c|}
\hline
$h $    & $dof $  & $\trb{\tQ_h u-u_h}$ & order & $\|Q_0u-u_0\|$ & order \\
\hline
   1/8   & 4.6500e+02 & 3.8193e-01 &       & 2.6130e-02 &  \\
\hline
   1/16  & 1.8250e+03 & 1.9065e-01 & 1.0024  & 6.5871e-03 & 1.9880  \\
\hline
   1/32  & 7.2330e+03 & 9.5281e-02 & 1.0006  & 1.6503e-03 & 1.9969  \\
\hline
   1/64  & 2.8801e+04 & 4.7635e-02 & 1.0002  & 4.1281e-04 & 1.9992  \\
\hline
1/128 & 1.1495e+05 & 2.3817e-02 & 1.0000  & 1.0322e-04 & 1.9998  \\
\hline
\end{tabular}
\end{table}

\begin{table}[!http]\label{table-2}
\centering \caption{k=1 not uniform convergence rates.}
\begin{tabular}{|c|c|c|c|c|c|c|c|c|c|c|c|c|c|}
\hline
$h $    & $dof $  & $\trb{\tQ_h u-u_h}$ & order & $\|Q_0u-u_0\|$ & order \\
\hline
   1/8   & 7.5900e+02 & 2.7721e-01 &       & 9.9566e-03 &  \\
\hline
   1/16  & 3.0850e+03 & 1.3806e-01 & 1.0056  & 2.5059e-03 & 1.9903  \\
\hline
   1/32  & 1.2259e+04 & 7.0399e-02 & 0.9717  & 6.0741e-04 & 2.0446  \\
\hline
   1/64  & 4.9521e+04 & 3.5100e-02 & 1.0041  & 1.5030e-04 & 2.0149  \\
\hline
1/128 & 1.9978e+05 & 1.7449e-02 & 1.0084  & 3.6918e-05 & 2.0254  \\
\hline
\end{tabular}
\end{table}

\textbf{Example 2} Consider the problem (\ref{problem-eq1})-
(\ref{problem-eq2}) on the unit square $\Omega=[0,1]\times [0,1]$.
The analytic solution $u$ is set to be $u=x(1-x) y(1-y)$, the
right-hand side side and the Dirichlet boundary condition is
computed to match the exact solution. In this numerical experiment,
the uniform rectangle mesh is employed. We can conclude from Table
\ref{table-4} that the convergence rate for $H^1$ and $L^2$ errors
are of $O(h)$ and $O(h^2)$, respectively. This is coincide with the
theoretical results.

\begin{table}[!http]\label{table-4}
\centering \caption{k=1 uniform rectangle mesh convergence rates.}
\begin{tabular}{|c|c|c|c|c|c|c|c|c|c|c|c|c|c|}
\hline
$h $    & $dof $  & $\trb{\tQ_h u-u_h}$ & order & $\|Q_0u-u_0\|$ & order \\
\hline
   1/8   & 2.7300e+02 & 2.9292e-02 &       & 1.8766e-03 &  \\
\hline
   1/16  & 1.0570e+03 & 1.4587e-02 & 1.0059  & 4.8858e-04 & 1.9415  \\
\hline
   1/32  & 4.1610e+03 & 7.2859e-03 & 1.0015  & 1.1926e-04 & 2.0344  \\
\hline
   1/64  & 1.6513e+04 & 3.6420e-03 & 1.0004  & 3.1107e-05 & 1.9389  \\
\hline
1/128 & 6.5793e+04 & 1.8209e-03 & 1.0001  & 7.5782e-06 & 2.0373  \\
\hline
\end{tabular}
\end{table}

\section{Appendix}
In this section, we shall give some technique tools which are applied
in the error estimate.

The trace inequality and the inverse inequality for the weak
Galerkin method are proved in \cite{WY14}.
\begin{lemma}
Assume that the partition $\T_h$ satisfies the assumptions (A1), (A2),
and (A3) as specified in \cite{WY14}. Then, there exists a constant $C$
 such that for any $T \in \T_h$ and
edge/face $e \in\partial T$ , we have for any $\theta\in H^1(T)$,
\begin{eqnarray}\label{trace-ineq}
\|\theta\|_e^p\le Ch_T^{-1}(\|\theta\|_T^p+h_T^p\|\nabla\theta\|_T^p).
\end{eqnarray}
\end{lemma}

\begin{lemma}
Assume that the partition $\T_h$ satisfies the assumptions (A1), (A2), (A3)
and (A4) as specified in \cite{WY14}. Then, there exists a constant $C(k)$
 such that for any $T \in \T_h$, we have for any $\varphi\in P_k(T)$,
\begin{eqnarray}\label{inverse-ineq}
\|\nabla\varphi\|_T\le C(n)h_T^{-1}\|\varphi\|_T.
\end{eqnarray}
\end{lemma}

The following estimate of interpolation error plays an
essential role in the error estimate.
\begin{lemma}\label{err-est1}
Assume that the partition $\T_h$ satisfies the assumptions (A1), (A2), (A3)
and (A4) as specified in \cite{WY14}.
For any $w\in H^{k+1}(\Omega)$, the following estimate holds true
\begin{eqnarray*}
\sumT\|w-I_b w\|^2_{\partial T}\le Ch^{2k+1}\|w\|^2_{k+1}.
\end{eqnarray*}
\end{lemma}
\begin{proof}
For any edge $e\in\partial T$, according to assumption (A3), there
exists a triangle $P_e$, whose base is identical to $e$ and height
is proportional to $h_T$. Then there exists a interpolation operator
$I_{Pe}$ onto $P_k(P_e)$ which contains all the interpolation points
of $I_b$. It follows the trace inequality (\ref{trace-ineq}) and the
property of interpolation operator that for any $w\in H^{k+2}(T)$,
\begin{eqnarray*}
&&\|w-I_b w\|^2_e
\\
&=&\|w-I_{Pe} w\|^2_e
\\
&\le& C(h_{P_e}^{-1}\|w-I_{Pe} w\|^2_{P_e}+h_{P_e}\|\nabla(w-I_{Pe} w)\|^2_{P_e})
\\
&\le& Ch_{P_e}^{2k+1}\|w\|_{k+1,P_e}^2
\\
&\le& Ch_T^{2k+1}\|w\|_{k+1,T}^2.
\end{eqnarray*}
Summing over all $T\in\T_h$ and the proof is completed.
\end{proof}

With Lemma \ref{err-est1}, we can describe the difference between
$\nabla_w\tQ_h w$ and $\dQ_h \nabla w$ in the following lemma.
\begin{lemma}\label{Qh-dQ}
Assume that the partition $\T_h$ satisfies the assumptions (A1), (A2), (A3)
and (A4) as specified in \cite{WY14}.
For any $w\in H^{k+1}(\Omega)$, the following estimate holds true
\begin{eqnarray*}
\|\nabla_w \tQ_h w-\dQ_h \nabla w\|\le C h^k\|w\|_{k+1}.
\end{eqnarray*}
\end{lemma}
\begin{proof}
It follows Lemma \ref{commu-prop} that
for any $\bq\in [P_{k-1}(T)]^2$,
\begin{eqnarray*}
&&(\nabla_w\tQ_h w-\dQ_h \nabla w,\bq)
\\
&=&\sumT\langle I_b w- w,\bq\cdot\bn\rangle
_{\partial T}
\\
&\le& C\left(\sumT \|I_b w- w\|^2
_{\partial T}\right)^\frac12
\left(\sumT \|\bq\|^2
_{\partial T}\right)^\frac12.
\end{eqnarray*}
From the trace inequality and the inverse inequality,
we can obtain
\begin{eqnarray*}
&&\sumT\|\bq\|^2_{\partial T}
\\
&\le& C\sumT h_T^{-1}\|\bq\|^2_T+C\sumT
h_T\|\nabla\bq\|^2_T
\\
&\le& C\sumT h_T^{-1}\|\bq\|^2_T.
\end{eqnarray*}
Taking $\bq=\nabla_w\tQ_h w-\dQ_h \nabla w$ and
applying Lemma \ref{err-est1} yield
\begin{eqnarray*}
\|\nabla_w \tQ_h w-\dQ_h \nabla w\|^2\le C h^k\|w\|_{k+1}
\|\nabla_w \tQ_h w-\dQ_h \nabla w\|,
\end{eqnarray*}
which completes the proof.
\end{proof}
\begin{lemma}\label{Qh-H2}
Assume that the partition $\T_h$ satisfies the assumptions (A1), (A2), (A3)
and (A4) as specified in \cite{WY14}.
For any $w\in H^{2}(\Omega)$, the following estimate holds true
\begin{eqnarray*}
\|\nabla_w \tQ_h w\|_1\le C\|w\|_2.
\end{eqnarray*}
\end{lemma}
\begin{proof}
When $w\in P_1(T)$, from definition \ref{wgradient} we know
$w=\tQ_h w$ and $\nabla_w \tQ_h w=\nabla w$, so that
\begin{eqnarray*}
\|\nabla_w \tQ_h w-\nabla w\|_1=0,\quad\forall w\in P_1(T).
\end{eqnarray*}
Then, from Bramble-Hilbert Theorem it follows that
\begin{eqnarray*}
\|\nabla_w \tQ_h w-\nabla w\|_1\le Ch\|w\|_2,\quad\forall
w\in H^2(T).
\end{eqnarray*}
From the triangle inequality we can obtain
\begin{eqnarray*}
&&\|\nabla_w \tQ_h w\|_1
\\
&\le&\|\nabla_w \tQ_h w-\nabla w\|_1+\|\nabla w\|_1
\\
&\le& C\|w\|_2,
\end{eqnarray*}
which completes the proof.
\end{proof}

In order to give the error estimate, we need to estimate the
remainder of the error equation in Lemma \ref{error-eqn}.

\begin{lemma}\label{err-est4}
Suppose $w\in H^{k+1}(\Omega)$ and $v_h\in V_h$, then the
following estimates hold true
\begin{eqnarray*}
(a\nabla_w \tQ_h w-\dQ_h (a\nabla w),\nabla_w v_h)
&\le& Ch^k\|w\|_{k+1}\trb{v_h},
\\
\sumT\langle
(\dQ_h(a\nabla w)-a\nabla w)\cdot\bn,v_0-v_b\rangle_{\partial T}
&\le& Ch^k\|w\|_{k+1}\trb{v_h},
\\
s(\tQ_hw,v_h)
&\le& Ch^k\|w\|_{k+1}\trb{v_h}.
\end{eqnarray*}
Moreover, we have
\begin{eqnarray*}
l(w,v_h)\le Ch^k\|w\|_{k+1}\trb{v_h}.
\end{eqnarray*}
\end{lemma}
\begin{proof}
For the first inequality, from the triangle inequality and
Lemma \ref{err-est1} it follows that
\begin{eqnarray*}
&&(a\nabla_w \tQ_h w-\dQ_h (a\nabla w),\nabla_w v_h)
\\
&\le& \|a\nabla_w \tQ_h w-\dQ_h (a\nabla w)\|\|\nabla_w v_h\|
\\
&\le& \|a\nabla_w \tQ_h w-a(\dQ_h \nabla w)\|\trb{v_h}+
 \|a(\dQ_h \nabla w)-a\nabla w\|\trb{v_h}
 \\
 &&+
  \|a\nabla w -\dQ_h (a\nabla w)\|\trb{v_h}
\\
&\le& Ch^k\|w\|_{k+1}\trb{v_h}.
\end{eqnarray*}
As the second inequality, from the properties of projection
operator and the trace inequality we can obtain
\begin{eqnarray*}
&&\sumT\langle
(\dQ_h(a\nabla w)-a\nabla w)\cdot\bn,v_0-v_b\rangle_{\partial T}
\\
&\le&C\left(\sumT h_T\|\dQ_h(a\nabla w)-a\nabla w\|_{\partial T}^2\right)
^\frac12
\left(\sumT h_T^{-1}\|v_0-v_b\|_{\partial T}^2\right)
^\frac12
\\
&\le&C\left(\sumT \|\dQ_h(a\nabla w)-a\nabla w\|_{T}^2\right)
^\frac12
\left(\sumT h_T^{-1}\|v_0-v_b\|_{\partial T}^2\right)
^\frac12
\\
&\le& Ch^k\|w\|_{k+1}\trb{v_h}.
\end{eqnarray*}
Similarly, for the last inequality we have
\begin{eqnarray*}
&&s(\tQ_hw,v_h)
\\
&=&\sumT h_T^{-1}\langle Q_0 w-I_b w,v_0-v_b\rangle_{\partial T}
\\
&=&\sumT h_T^{-1}\langle Q_0 w- w,v_0-v_b\rangle_{\partial T}
+\sumT h_T^{-1}\langle w-I_b w,v_0-v_b\rangle_{\partial T}
\\
&\le&C\left(\sumT h_T^{-1}\|Q_0 w- w\|_{\partial T}^2\right)
^\frac12
\left(\sumT h_T^{-1}\|v_0-v_b\|_{\partial T}^2\right)
^\frac12
\\
&&+C\left(\sumT h_T^{-1}\| w-I_b w\|_{\partial T}^2\right)
^\frac12
\left(\sumT h_T^{-1}\|v_0-v_b\|_{\partial T}^2\right)
^\frac12
\\
&\le& Ch^k\|w\|_{k+1}\trb{v_h}.
\end{eqnarray*}
\end{proof}

\end{document}